\documentclass{amsart}

\usepackage{amsfonts,amsmath,amssymb,amsthm,mathrsfs,amsxtra}
\usepackage{enumerate,verbatim}

\usepackage[all,2cell,ps]{xy}
\usepackage{hyperref}

\DeclareMathOperator{\ann}{ann}

\DeclareMathOperator{\depth}{depth}
\DeclareMathOperator{\Ext}{Ext}

\DeclareMathOperator{\Hom}{Hom}
\DeclareMathOperator{\Image}{Image}
\DeclareMathOperator{\injdim}{injdim}

\DeclareMathOperator{\rank}{rank}
\DeclareMathOperator{\Soc}{Soc}

\renewcommand{\ge}{\geqslant}

\theoremstyle{plain}
\newtheorem{theorem}{Theorem}[section]
\newtheorem{lemma}[theorem]{Lemma}
\newtheorem{proposition}[theorem]{Proposition}
\newtheorem{corollary}[theorem]{Corollary}

\newtheorem{innercustomthm}{Theorem}
\newenvironment{customthm}[1]
  {\renewcommand\theinnercustomthm{#1}\innercustomthm}
  {\endinnercustomthm}
  
\theoremstyle{definition}
\newtheorem{definition}[theorem]{Definition}
\newtheorem{example}[theorem]{Example}
\newtheorem{question}[theorem]{Question}

\theoremstyle{remark}
\newtheorem{remark}[theorem]{Remark}

\numberwithin{equation}{theorem}

\title[Characterizations of regular local rings]{Characterizations of regular local rings via syzygy modules of the residue field}

\author[D. Ghosh]{Dipankar Ghosh}
\address{Department of Mathematics, Indian Institute of Technology Bombay, Powai, Mumbai 400076, India}
\email{dipankar@math.iitb.ac.in}
\author[A. Gupta]{Anjan Gupta}
\email{anjan@math.iitb.ac.in}

\author[T. J. Puthenpurakal]{Tony J. Puthenpurakal}
\email{tputhen@math.iitb.ac.in}

\begin{document}

\begin{abstract}
 Let $R$ be a commutative Noetherian local ring with residue field $k$. We show that if a finite direct sum of syzygy modules of $k$
 surjects onto `a semidualizing module' or `a non-zero maximal Cohen-Macaulay module of finite injective dimension', then $R$ is
 regular. We also prove that $R$ is regular if and only if some syzygy module of $k$ has a non-zero direct summand of finite
 injective dimension.
\end{abstract}

\maketitle

\section{Introduction}\label{Introduction}
 Throughout this article, unless otherwise specified, all rings are assumed to be commutative Noetherian local rings, and all modules
 are assumed to be finitely generated.

 In this article, $R$ always denotes a  local ring with maximal ideal $\mathfrak{m}$ and residue field $k$.
 Let $\Omega_n^R(k)$ be the $n$th syzygy module of $k$. Dutta gave the following characterization of regular local rings in
 \cite[Corollary~1.3]{Dut89}.

\begin{theorem}[Dutta]\label{theorem: Dutta}
 $R$ is regular if and only if $\Omega_n^R(k)$ has a non-zero free direct summand for some integer $n \ge 0$.
\end{theorem}

 Later, Takahashi generalized Dutta's result by giving a characterization of regular local rings via the existence of a semidualizing
 direct summand of some syzygy module of the residue field. Let us recall the definition of a semidualizing module.

\begin{definition}[\cite{Gol84}]\label{definition: semidualizing module}
 An $R$-module $M$ is said to be a {\it semidualizing module} if the following hold:
 \begin{enumerate}[(i)]
  \item The natural homomorphism $R \to \Hom_R(M,M)$ is an isomorphism.
  \item $\Ext_R^i(M,M) = 0$ for all $i \ge 1$.
 \end{enumerate}
\end{definition}

 Note that $R$ itself is a semidualizing $R$-module. So the following theorem generalizes the above result of Dutta.

\begin{theorem}\cite[Theorem~4.3]{Tak06}\label{theorem: Takahashi; regular; semidualizing}
 $R$ is regular if and only if $\Omega_n^R(k)$ has a semidualizing direct summand for some integer $n \ge 0$.
\end{theorem}

 If $R$ is a Cohen-Macaulay local ring with canonical module $\omega$, then $\omega$ is a semidualizing $R$-module. Therefore,
 as an application of the above theorem, Takahashi obtained the following:

\begin{corollary}\cite[Corollary~4.4]{Tak06}\label{corollary: Takahashi; regular; canonical module}
 Let $R$ be a Cohen-Macaulay local ring with canonical module $\omega$. Then $R$ is regular if and only if $\Omega_n^R(k)$ has a
 direct summand isomorphic to $\omega$ for some integer $n \ge 0$.
\end{corollary}

Now recall that the canonical module (if exists) over a Cohen-Macaulay local ring has finite injective dimension. Also it is
well-known that $R$ is regular if and only if $k$ has finite injective dimension. So, in this direction, a natural question arises
that ``if $\Omega_n^R(k)$ has a non-zero direct summand of finite injective dimension for some integer $n \ge 0$, then is the ring
$R$ regular?''. In the present study, we see that the above question has an affirmative answer.

Kaplansky conjectured that if some power of the maximal ideal of $R$ is non-zero and of finite projective dimension, then $R$ is
regular. In \cite[Theorem~1.1]{LV68}, Levin and Vasconcelos proved this conjecture. In fact, their result is even stronger:

 \begin{theorem}\label{theorem: Levin and Vasconcelos}
  If $M$ is an $R$-module such that $\mathfrak{m} M$ is non-zero and of finite projective dimension {\rm (}or of
  finite injective dimension{\rm )}, then $R$ is regular.
 \end{theorem}
 
 Motivated by this theorem, Martsinkovsky generalized Dutta's result in the following direction. {\it We denote a finite
 collection of non-negative integers by $\Lambda$}. In \cite[Proposition~7]{Mar96}, Martsinkovsky proved the following:
 
 \begin{theorem}\label{theorem: Martsinkovsky}
  Let $f : \bigoplus_{n \in \Lambda} \left( \Omega_n^R(k) \right)^{j_n} \longrightarrow L$\quad $(j_n \ge 1$ for each $n \in \Lambda)$
  be a surjective $R$-module homomorphism, where $L$ is non-zero and of finite projective dimension. Then $R$ is regular.
 \end{theorem}
 
 For a stronger result, we refer the reader to \cite[Corollary~9]{Avr96}. In this direction, we prove the following result which
 considerably strengthens Theorem~\ref{theorem: Takahashi; regular; semidualizing}. The proof presented here is very simple and
 elementary.
 
 \begin{customthm}{I}[See Corollary~\ref{corollary: RLR and surjection onto semidualizing}]
                     \label{Theorem I: surjection onto semidualizing}
  If a finite direct sum of syzygy modules of $k$ surjects onto a semidualizing $R$-module, then $R$ is regular.
 \end{customthm}
 
 Furthermore, we raise the following question:
 
 \begin{question}\label{question: surjection onto finite injdim}
  If a finite direct sum of syzygy modules of $k$ surjects onto a non-zero $R$-module of finite injective dimension,
  then is the ring $R$ regular?
 \end{question}
 
 In this article, we give a partial answer to this question as follows:
 
 \begin{customthm}{II}[See Corollary~\ref{corollary: RLR and surjection onto finite injdim}]
		       \label{Theorem II: surjection onto semidualizing}
  If a finite direct sum of syzygy modules of $k$ surjects onto a non-zero maximal Cohen-Macaulay $R$-module $L$ of finite
  injective dimension, then $R$ is regular.
 \end{customthm}

 If $R$ is a Cohen-Macaulay local ring with canonical module $\omega$, then one can take $L = \omega$ in the above theorem.
 
%  As a consequence of this theorem, we obtain a new characterization of regular local rings:
%%%  New Sentences %%%
 We obtain one new characterization of regular local rings. It follows from Dutta's result (Theorem~\ref{theorem: Dutta}) that $A$
 is regular if and only if some syzygy module of $k$ has a non-zero direct summand of finite projective dimension. Here we prove the
 following counterpart for injective dimension.
%%%  New Sentences %%%
 
 \begin{customthm}{III}[See Theorem~\ref{theorem: characterization of RLR, injdim}]
  $R$ is regular if and only if some syzygy module of $k$ has a non-zero direct summand of finite injective dimension.
 \end{customthm}

 Moreover, this result has a dual companion; see Corollary~\ref{corollary: dual result, cosyzygy}.
 
 Till now we have considered surjective homomorphisms from a finite direct sum of syzygy modules of $k$ to a `special module'.
 One may ask ``what happens if we consider injective homomorphisms from a `special module' to a finite direct sum of syzygy modules
 of $k$?". More precisely, if
 \[
  f : L \longrightarrow \bigoplus_{n \in \Lambda} {\left( \Omega_n^R(k) \right)}^{j_n}
 \]
 %%% ``injective homomorphism of finitely generated $R$-modules" changed to ``injective $R$-module homomorphism"
 is an injective $R$-module homomorphism, where $L$ is non-zero and of finite projective dimension (or of finite injective dimension),
 then is the ring $R$ regular? In this situation, we show that $R$ is not necessarily a regular local
 ring; see Example~\ref{example: counter, injective map}.

\section{Preliminaries}\label{Preliminaries}
 
 In the present section, we give some preliminaries which we use in order to prove our main results. We start with the following
 lemma which gives a relation between the socle of the ring and the annihilator of the syzygy modules.
 
 \begin{lemma}\label{lemma: socle, syzygy}
  Let $M$ be an $R$-module. Then, for every integer $n \ge 1$, we have
  \[
   \Soc(R) \subseteq \ann_R\left(\Omega_n^R(M)\right).
  \]
  In particular, if $R \neq k$ {\rm (}i.e., if  $\mathfrak{m} \neq 0${\rm )}, then
  \[
   \Soc(R) \subseteq \ann_R\left(\Omega_n^R(k)\right) \mbox{ for all } n \ge 0.
  \]
 \end{lemma}
 
 \begin{proof}
  Fix $n \ge 1$. If $\Omega_n^R(M) = 0$, then we are done. So we may assume $\Omega_n^R(M) \neq 0$. Consider the following commutative
  diagram in the minimal free resolution of $M$:
  \[
  \xymatrixrowsep{8mm} \xymatrixcolsep{6mm}
  \xymatrix{ \cdots \ar[r] & R^{b_n} \ar@{->>}[rd]_{f} \ar[rr]^{\delta} 	& 			& R^{b_{n-1}}\ar[r] & \cdots. \\
			    &                                     	& \Omega_n^R(M)\ar@{^{(}->}[ru]_g	}
  \]
  Let $a \in \Soc(R)$, i.e., $a\mathfrak{m} = 0$. Suppose $x \in \Omega_n^R(M)$. Since $f$ is surjective, there exists
  $y \in R^{b_n}$ such that $f(y) = x$. Note that $\delta(ay) = a\delta(y) = 0$ as
  $\delta\left(R^{b_n}\right) \subseteq \mathfrak{m} R^{b_{n-1}}$ and $a\mathfrak{m} = 0$.  Therefore
  $g(ax) = g(f(ay)) = \delta(ay) = 0$, which gives $ax = 0$ as $g$ is injective.
  Thus $\Soc(R) \subseteq \ann_R\left(\Omega_n^R(M)\right)$.
  
  For the last part, note that $\Soc(R) \subseteq \mathfrak{m} = \ann_R\left(\Omega_0^R(k)\right)$ if $\mathfrak{m} \neq 0$.
 \end{proof}
 
 Let us recall the following well-known result initially obtained by Nagata.
 
 \begin{proposition}\cite[Corollary~5.3]{Tak06}\label{proposition: Nagata}
  Let $x \in \mathfrak{m} \smallsetminus \mathfrak{m}^2$ be an $R$-regular element.
  Set $\overline{(-)} := (-) \otimes_R R/(x)$. Then
  \[ 
    \overline{\Omega_n^{R}(k)} \cong \Omega_n^{\overline{R}}(k)\oplus \Omega_{n-1}^{\overline{R}}(k)
    \quad \mbox{for every integer $n \ge 1$}.
  \]
 \end{proposition}

 We notice two properties satisfied by semidualizing modules and maximal Cohen-Macaulay modules of finite injective dimension.
 
 \begin{definition}\label{definition: star property}
  %%% `property' changed to property
  Let $\mathcal{P}$ be a property of modules over  local rings. We say $\mathcal{P}$ is a $(*)$-property if $\mathcal{P}$
  satisfies the following:
  \begin{enumerate}[(1)]
   \item An $R$-module $M$ satisfies $\mathcal{P}$ implies that the $R/(x)$-module $M/xM$ satisfies $\mathcal{P}$, where $x \in R$
         is an $R$-regular element.
   \item An $R$-module $M$ satisfies $\mathcal{P}$ and $\depth(R) = 0$ together imply that $\ann_R(M) = 0$.
  \end{enumerate}
 \end{definition}
 
 Now we give a few examples of $(*)$-properties.
  
 \begin{example}\label{example: star property: semidualizing}
  The property $\mathcal{P}$ $:=$ `semidualizing modules over  local rings' is a $(*)$-property.
  %%%  ``semidualizing modules over Noetherian local rings" changed to `semidualizing modules over Noetherian local rings' everywhere
  %%%%%%  $``P"$ changed to $\mathcal{P}$
 \end{example}
 
 \begin{proof}
  Let $C$ be a semidualizing $R$-module. It is shown in \cite[p.~68]{Gol84} that $C/xC$ is a semidualizing $R/(x)$-module, where
  $x \in \mathfrak{m}$ is an $R$-regular element. Since $\Hom_R(C,C) \cong R$, we have $\ann_R(C) = 0$
  (without any restriction on $\depth(R)$).
 \end{proof}
 
 Here is another example of $(*)$-property.
 
 \begin{example}\label{example: star property: MCM, finite injdim}
  The property $\mathcal{P}$ $:=$ `non-zero maximal Cohen-Macaulay modules of finite injective dimension over Cohen-Macaulay local
  rings' is a $(*)$-property.
  %%%  ``non-zero ... " changed to `non-zero ... '
 \end{example}
 
 \begin{proof}
  Let $R$ be a Cohen-Macaulay local ring, and let $L$ be a non-zero maximal Cohen-Macaulay $R$-module of finite injective dimension.
  Suppose $x \in R$ is an $R$-regular element. Since $L$ is a maximal Cohen-Macaulay $R$-module, $x$ is $L$-regular as well. Therefore
  $L/xL$ is a non-zero maximal Cohen-Macaulay module of finite injective dimension over the Cohen-Macaulay local ring $R/(x)$
  (see \cite[3.1.15]{BH98}).
  
  Now further assume that $\depth(R) = 0$. Then $R$ is an Artinian local ring, and
  \[
   \injdim_R(L) = \depth(R) = 0.
  \] %%% changed to \[ --- \], inserted ,
  Therefore, by \cite[3.2.8]{BH98}, we have that $L \cong E^r$, where $E$ is the injective hull of $k$ and $r = \rank_k\left(\Hom_R(k,L)\right)$.
  It is well-known that $\Hom_R(E, E) \cong R$ as $R$ is an Artinian local ring. Hence $\ann_R(L) = \ann_R(E) = 0$.
 \end{proof}
 
 \section{Main results}\label{Main results}
 
 Now we are in a position to prove our main results. First of all, we prove that if a finite direct sum of syzygy modules of the
 residue field surjects onto a non-zero module satisfying a $(*)$-property, then the ring is regular.
 
 \begin{theorem}\label{theorem: RLR and surjection onto star module}
  %%%% `that' included
  Assume that $\mathcal{P}$ is a $(*)$-property {\rm (}see {\rm Definition~\ref{definition: star property})}. Let
  %%% $(R,\mathfrak{m},k)$ be a Noetherian local ring, and let
  \[
    f : \bigoplus_{n\in\Lambda} {\left( \Omega_n^R(k) \right)}^{j_n} \longrightarrow L
    \quad\quad\mbox{$(j_n \ge 1$ for each $n \in \Lambda)$}
  \]
  be a surjective $R$-module homomorphism, where $L$ {\rm(}$\neq 0${\rm)} satisfies $\mathcal{P}$. Then $R$ is regular.
 \end{theorem}
 
 \begin{proof}
  We prove the theorem by induction on $t := \depth(R)$. Let us first assume $t = 0$. If possible, let $R \neq k$, i.e.,
  $\mathfrak{m} \neq 0$. Since $\depth(R) = 0$, we have $\Soc(R) \neq 0$. But, by virtue of Lemma~\ref{lemma: socle, syzygy}, we obtain
  \begin{align*}
   \Soc(R) & \subseteq \bigcap_{n \in \Lambda} \ann_R\left(\Omega_n^R(k)\right) \\
           & = \ann_R\left( \bigoplus_{n \in \Lambda} {\left( \Omega_n^R(k) \right)}^{j_n} \right) \\
           & \subseteq \ann_R(L) \quad\quad
                    \mbox{[as $f : \bigoplus_{n\in\Lambda} {\left( \Omega_n^R(k) \right)}^{j_n} \longrightarrow L$ is surjective]}\\
           & = 0 \quad\quad\quad \mbox{[as $L$ satisfies $\mathcal{P}$ which is a $(*)$-property]},
  \end{align*}
  which is a contradiction. Therefore $R$ ($= k$) is a regular local ring.
  
  Now we assume $t \ge 1$. Suppose the theorem holds true for all such rings of depth smaller than $t$. Since
  $\depth(R) \ge 1$, there exists an element $x \in \mathfrak{m} \smallsetminus \mathfrak{m}^2$ which is $R$-regular. We set
  $\overline{(-)} := (-) \otimes R/(x)$. Clearly,
  \[
    \overline{f} : \bigoplus_{n\in\Lambda} {\left( \overline{\Omega_n^R(k)} \right)}^{j_n} \longrightarrow  \overline{L}
  \]
  is a surjective $\overline{R}$-module homomorphism, where the $\overline{R}$-module $\overline{L}$ {\rm(}$\neq 0${\rm)}
  satisfies $\mathcal{P}$ as $\mathcal{P}$ is a $(*)$-property. Since $x \in \mathfrak{m} \smallsetminus \mathfrak{m}^2$ is an $R$-regular
  element, by Proposition~\ref{proposition: Nagata}, we have
  \[ 
     \bigoplus_{n \in \Lambda} \left( \overline{\Omega_n^R(k)} \right)^{j_n} \cong
     \bigoplus_{n\in\Lambda} \left( \Omega_n^{\overline{R}}(k) \oplus \Omega_{n-1}^{\overline{R}}(k) \right)^{j_n}
     \quad \mbox{[by setting $\Omega_{-1}^{\overline{R}}(k) := 0$]}.
  \]
  Since $\depth(\overline{R}) = \depth(R) - 1$, by the induction hypothesis, $\overline{R}$ is a regular local ring, and hence $R$
  is a regular local ring as $x \in \mathfrak{m} \smallsetminus \mathfrak{m}^2$.
 \end{proof}
 
 As a few applications of the above theorem, we obtain the following necessary and sufficient conditions for a  local ring
 to be regular. %%% First we prove the following:
 
 \begin{corollary}\label{corollary: RLR and surjection onto semidualizing}
  Let $f : \bigoplus_{n \in \Lambda} {\left( \Omega_n^R(k) \right)}^{j_n} \longrightarrow L$
  be a surjective $R$-module homomorphism, where $L$ is a semidualizing $R$-module. Then $R$ is regular.
 \end{corollary}
 
 \begin{proof}
  The corollary follows from Theorem~\ref{theorem: RLR and surjection onto star module}
  and Example~\ref{example: star property: semidualizing}.
 \end{proof}

 \begin{remark}\label{remark: Dutta and Takahashi's result}
  We can recover Theorem~\ref{theorem: Takahashi; regular; semidualizing} (in particular, Theorem~\ref{theorem: Dutta} because $R$
  itself is a semidualizing $R$-module) as a consequence of Corollary~\ref{corollary: RLR and surjection onto semidualizing}.
  In fact the above result is even stronger than Theorem~\ref{theorem: Takahashi; regular; semidualizing}.
 \end{remark}
 
 Now we give a partial answer to Question~\ref{question: surjection onto finite injdim}.
 
 \begin{corollary}\label{corollary: RLR and surjection onto finite injdim}
  Let $(R,\mathfrak{m},k)$ be a Cohen-Macaulay local ring, and let
  \[ 
    f : \bigoplus_{n \in \Lambda} {\left( \Omega_n^R(k) \right)}^{j_n} \longrightarrow L
  \]
  be a surjective $R$-module homomorphism, where $L$ {\rm (}$\neq 0${\rm )} is a maximal Cohen-Macaulay $R$-module of finite injective
  dimension. Then $R$ is regular.
 \end{corollary}
 
 \begin{proof}
  The corollary follows from Theorem~\ref{theorem: RLR and surjection onto star module}
  and Example~\ref{example: star property: MCM, finite injdim}.
 \end{proof}
 
 \begin{remark}\label{remark: partial answer for Artinian rings}
  It is clear from the above corollary that Question~\ref{question: surjection onto finite injdim} has an affirmative answer for
  Artinian local rings.
 \end{remark}

 Let $R$ be a Cohen-Macaulay local ring. Recall that a maximal Cohen-Macaulay $R$-module $\omega$ of type $1$ and of finite injective
 dimension is called the {\it canonical module} of $R$. It is well-known that the canonical module $\omega$ of $R$ is a semidualizing
 $R$-module; see \cite[3.3.10]{BH98}. So, both Corollary~\ref{corollary: RLR and surjection onto semidualizing}
 and Corollary~\ref{corollary: RLR and surjection onto finite injdim} yield the following result (independently) which
 strengthens Corollary~\ref{corollary: Takahashi; regular; canonical module}.
 
 \begin{corollary}\label{corollary: RLR and surjection onto omega}
  Let $(R,\mathfrak{m},k)$ be a Cohen-Macaulay local ring with canonical module $\omega$, and let
  $f : \bigoplus_{n \in \Lambda} \left( \Omega_n^R(k) \right)^{j_n} \longrightarrow \omega$ be a surjective $R$-module homomorphism.
  Then $R$ is regular.
 \end{corollary}
 
 Here we obtain one new characterization of regular local rings. The following characterization is based on the existence of a
 non-zero direct summand with finite injective dimension of some syzygy module of the residue field.
 
\begin{theorem}\label{theorem: characterization of RLR, injdim}
 The following statements are equivalent: %%% Inserted ``statements"
 \begin{enumerate}[{\rm (1)}]
  \item $R$ is a regular local ring.
  \item $\Omega_n^R(k)$ has a non-zero direct summand of finite injective dimension for some integer $n \ge 0$.
 \end{enumerate}
\end{theorem}

\begin{proof}
 (1) $\Rightarrow$ (2). If $R$ is regular, then $\Omega_0^R(k)$ $(= k)$ itself is a non-zero $R$-module of finite injective
 dimension. Hence the implication follows.
 
 (2) $\Rightarrow$ (1). Without loss of generality, we may assume that $R$ is complete. Existence of a non-zero (finitely generated)
 %%% Inserted bracket 
 $R$-module of finite injective dimension ensures that $R$ is a Cohen-Macaulay local ring (see \cite[9.6.2 and 9.6.4(ii)]{BH98} and
 \cite{Rob87}). Therefore we may as well assume that $R$ is a Cohen-Macaulay complete local ring.
 
 Suppose that $L$ is a non-zero direct summand of $\Omega_n^R(k)$ with $\injdim_R(L)$ finite for some integer $n \ge 0$. We prove the
 implication by induction on $d := \dim(R)$. If $d = 0$, then the implication follows from
 Corollary~\ref{corollary: RLR and surjection onto finite injdim}.
 
 Now we assume $d \ge 1$. Suppose the implication holds true for all such rings of dimension smaller than $d$. Since $R$ is
 Cohen-Macaulay and $\dim(R) \ge 1$, there exists $x \in \mathfrak{m} \smallsetminus \mathfrak{m}^2$ which is $R$-regular.
 We set $\overline{(-)} := (-) \otimes_R R/(x)$. If $n = 0$, then the direct summand $L$ of $\Omega_0^R(k)$ $(= k)$ must be equal to
 $k$, and hence $\injdim_R(k)$ is finite, which gives $R$ is regular. Therefore we may assume $n \ge 1$. Hence $x$ is
 $\Omega_n^R(k)$-regular. Since $L$ is a direct summand of $\Omega_n^R(k)$, $x$ is $L$-regular as well. This gives
 $\injdim_{\overline{R}}(\overline{L})$ is finite. Now we fix an indecomposable direct summand $L'$ of $\overline{L}$. Then
 $\injdim_{\overline{R}}(L')$ is also finite. Note that the $\overline{R}$-module $\overline{L}$ is a direct summand
 of $\overline{\Omega_n^R(k)}$. Hence $L'$ is an indecomposable direct summand of $\overline{\Omega_n^R(k)}$.
 Since $x \in \mathfrak{m} \smallsetminus \mathfrak{m}^2$ is an $R$-regular element, by Proposition~\ref{proposition: Nagata}, we have
 \begin{equation*}
  \overline{\Omega_n^{R}(k)} \cong \Omega_n^{\overline{R}}(k)\oplus \Omega_{n-1}^{\overline{R}}(k).
 \end{equation*}
 It then follows from the uniqueness of Krull-Schmidt decomposition (\cite[Theorem~(21.35)]{Lam01}) that $L'$ is isomorphic
 to a direct summand of $\Omega_n^{\overline{R}}(k)$ or $\Omega_{n-1}^{\overline{R}}(k)$. Since $\dim(\overline{R}) = \dim(R) - 1$,
 by the induction hypothesis, $\overline{R}$ is a regular local ring, and hence $R$ is a regular local ring
 as $x \in \mathfrak{m} \smallsetminus \mathfrak{m}^2$.
\end{proof}
 
 Let $M$ be an $R$-module. Consider the augmented minimal injective resolution of $M$:
 \[
  0 \longrightarrow M \stackrel{d^{-1}}{\longrightarrow} I^0 \stackrel{d^0}{\longrightarrow} I^1 \stackrel{d^1}{\longrightarrow}
  I^2 \longrightarrow \cdots \longrightarrow I^{n-1} \stackrel{d^{n-1}}{\longrightarrow} I^n \longrightarrow \cdots.
 \]
 Recall that the $n$th {\it cosyzygy module} of $M$ is defined by
 \[
  \Omega_{-n}^R(M) := \Image(d^{n-1}) \quad\mbox{for all } n \ge 0.
 \]
 
 The following result is dual to Theorem~\ref{theorem: characterization of RLR, injdim} which gives another characterization of
 regular local rings via cosyzygy modules of the residue field.
 
 \begin{corollary}\label{corollary: dual result, cosyzygy}
  The following statements are equivalent: %%% Inserted ``statements"
  \begin{enumerate}[{\rm (1)}]
   \item $R$ is a regular local ring.
   \item Cosyzygy module $\Omega_{-n}^R(k)$ has a non-zero finitely generated direct summand of finite projective
         dimension for some integer $n \ge 0$.
  \end{enumerate}
 \end{corollary}
 
 \begin{proof}
  (1) $\Rightarrow$ (2). If $R$ is regular, then $\Omega_0^R(k)$ $(= k)$ has finite projective dimension. Hence the implication
  follows.
  
  (2) $\Rightarrow$ (1). Without loss of generality, we may assume that $R$ is complete. Suppose $\Omega_{-n}^R(k) \cong P \oplus Q$
  for some integer $n \ge 0$, where $P$ is a non-zero finitely generated $R$-module of finite projective dimension. Consider the
  following part of the minimal injective resolution of $k$:
  \begin{equation}\label{theorem: dual result, cosyzygy: equation 1}
   0 \to k \to E \to E^{\mu_1} \to \cdots \to E^{\mu_{n-1}}
   \to \Omega_{-n}^R(k) \cong P \oplus Q \to 0,
  \end{equation}
  where $E$ is the injective hull of $k$. Dualizing \eqref{theorem: dual result, cosyzygy: equation 1} with respect to $E$ and
  using $\Hom_R(k,E) \cong k$ and $\Hom_R(E,E) \cong R$ (cf. \cite[3.2.12(a) and 3.2.13(a)]{BH98}), we get the following part of the
  minimal free resolution of $k$:
  \[ 0 \to \Hom_R(P,E) \oplus \Hom_R(Q,E) \cong \Omega_n^R(k) \to R^{\mu_{n-1}} \to \cdots
       \to R^{\mu_1} \to R \to k \to 0. \]
  Clearly, $\Hom_R(P,E)$ is non-zero and of finite injective dimension as $P$ is non-zero and of finite projective dimension.
  Therefore the implication follows from Theorem~\ref{theorem: characterization of RLR, injdim}.
 \end{proof}
 
 Now we give an example to ensure that the existence of an injective homomorphism from a `special module' to a finite direct sum of
 syzygy modules of the residue field does not necessarily imply that the ring is regular.
 
 \begin{example}\label{example: counter, injective map}
  Let $(R,\mathfrak{m},k)$ be a Gorenstein local domain of dimension $d$. Then $\Omega_d^R(k)$ is a maximal Cohen-Macaulay $R$-module;
  see, e.g., \cite[1.3.7]{BH98}. Since $R$ is Gorenstein, $\Omega_d^R(k)$ is a reflexive $R$-module (by \cite[3.3.10]{BH98}), and
  hence it is torsion-free. Then, by mapping $1$ to a non-zero element of $\Omega_d^R(k)$, we get an injective $R$-module homomorphism
  $f : R \to \Omega_d^R(k)$. Note that $\injdim_R(R)$ is finite. But a Gorenstein local domain need not be a regular local ring.
 \end{example}

\section*{Acknowledgements}
 The first author would like to thank NBHM, DAE, Govt. of India for providing financial support for this study.

\end{document}